\documentclass[a4paper,12pt]{article}
\usepackage[utf8]{inputenc}
\usepackage{amsmath}
\usepackage{amsthm,amssymb,amsfonts}
\usepackage{graphicx}

\usepackage{array}
\usepackage{color}
\usepackage{comment}
\usepackage[normalem]{ulem}

\newtheorem{theorem}{Theorem}
\newtheorem{definition}[theorem]{Definition}

\newtheorem{proposition}[theorem]{Proposition}
\newtheorem{remark}[theorem]{Remark}
\newtheorem{example}[theorem]{Example}
\newtheorem{lemma}[theorem]{Lemma}

\def\supp{{\mathop{\rm supp\, }}}
\def\Id{\rm Id}
\def\Sch{\mathcal{S}}
\def\R{\mathbb{R}}

\usepackage{scalerel,stackengine}
\stackMath
\newcommand\widecheck[1]{%
	\savestack{\tmpbox}{\stretchto{%
			\scaleto{%
				\scalerel*[\widthof{\ensuremath{#1}}]{\kern-.6pt\bigwedge\kern-.6pt}%
				{\rule[-\textheight/2]{1ex}{\textheight}}
			}{\textheight}%
		}{0.5ex}}%
	\stackon[1pt]{#1}{\scalebox{-1}{\tmpbox}}%
}
\title{Detecting quasicrystals with quadratic time-frequency distributions}
\author{P. Boggiatto, C. Fern\'andez, A. Galbis, A. Oliaro\footnote{The research of C. Fernández and A. Galbis was partially supported by the projects MCIN PID2020-119457GBI00/AEI/10.13039/501100011033 and GV Prometeu/2021/070. P. Boggiatto and A. Oliaro are members of INdAM, gruppo GNAMPA - Gruppo Nazionale per l'Analisi Matematica, la Probabilità e le loro Applicazioni, sezione Analisi Funzionale e Armonica.}}
\date{}

\begin{document}
	
\maketitle

\begin{abstract}
	The usefulness of time-frequency analysis methods in the study of quasicrystals was pointed out in \cite{bfgo}, where we proved that a tempered distribution $\mu$ on ${\mathbb R}^d$ whose Wigner transform is a measure supported on the cartesian product of two uniformly discrete sets in ${\mathbb R}^d$ is a Fourier quasicrystal. In this paper we go further in this direction using the matrix-Wigner transforms to detect quasicrystal structures. The results presented here cover essentially all the most important quadratic time-frequency distributions, and are obtained considering
	two different (disjoint) classes of matrix-Wigner transforms, discussed respectively in Theorems \ref{thm-main1} and \ref{thm-main2}. The transforms considered in Theorem \ref{thm-main1} include the classical Wigner transform, as well as all the time-frequency representations of matrix-Wigner type belonging to the Cohen class. On the other hand Theorem \ref{thm-main2}, which does not apply to the classical Wigner, has, as main example, the Ambiguity function. In this second case we prove a stronger result with respect to \cite{bfgo}, since we only suppose that the support of the matrix-Wigner transform of $\mu$ is contained in the cartesian product of two discrete sets, obtaining that both the support and the spectrum of $\mu$ are uniformly discrete.
\end{abstract}

\section{Introduction}

Quasicrystals are alloys whose diffraction spectrum consists of bright spots but which lacks the periodic structure of crystals. They were discovered experimentally in the mid-80s and their structure can be modeled using what we know today as {\it Fourier quasicrystals} or, more generally, {\it crystalline measures}. Following the approach of Lev and Olevskii \cite{olevskii2} we will call {\it Fourier quasicrystal} a tempered measure $\mu$ of the form $\mu = \sum_{\lambda\in \Lambda}a_\lambda \delta_\lambda$ for which $\widehat{\mu} = \sum_{s\in S}b_s \delta_s,$ where $\delta_\xi$ is the Dirac mass point at $\xi$, and $\Lambda$, $S$ are discrete subsets of ${\mathbb R}^n$, supposing further that $|\mu|$ and $|\widehat{\mu}|$ are also tempered distributions. If instead $\mu$ and $\widehat{\mu}$ are measures with locally finite support, $\mu$ is said to be a { \it cristalline measure}. Fourier quasicrystals and/or cristalline measures have been subject of a considerable literature, see for instance \cite{favorov, favorov2,favorov3,lagarias,olevskii2,olevskii1,Lev-Reti,olevskii3} and the references therein.
                                                                                                                                                                                                                                                                                            
The Wigner transform $W(\mu)$ describes the content of a tempered distribution $\mu$ simultaneously with respect to time and frequency, therefore it is reasonable to investigate if a quasicrystal structure of $\mu$ can be deduced from properties of $W(\mu)$ instead from separate conditions on $\mu$ and $\widehat \mu$, and this goal was achieved in \cite{bfgo}. On the other hand, there are many other time-frequency representations connected to the classical Wigner, and one reasonable question is whether the fact that $\mu$ is a quasicrystal can be deduced from features of time-frequency representations different from the Wigner, and if there are classes of representations that do this job better than others. The matrix-Wigner transform (cf. \cite{Bayer}) is defined as 
$$
W_T(\mu)(x,\omega)=\int_{\R^d} e^{-2\pi it\omega} \mu(A_0 x+B_0t)\overline{\mu(C_0x+D_0t)}\,dt,
$$
for $\mu\in\Sch(\R^d)$, with standard extension to the space of tempered distributions, where 
$$
T=\left(\begin{array}{cc} A_0 &B_0 \\ C_0 &D_0 \end{array}\right)
$$ 
is invertible and $A_0, B_0, C_0, D_0$ are $d\times d$ real matrices. 
The classical Wigner transform, denoted by $W(\mu)$, is obtained by choosing $A_0=C_0=\Id$, $B_0=\frac{1}{2}\Id$ and $D_0=-\frac{1}{2}\Id,$ where $\Id$ denotes the identity $d\times d$ matrix. Another widely used class of time-frequency representations is the so-called Cohen class (see \cite{Cohen66, Jans97}), defined as all sesquilinear forms of the type
$$
Q(\mu)=\sigma*W(\mu)
$$
for $\sigma\in\Sch'(\R^{2d})$. Every covariant and weakly continuous quadratic time-frequency distribution belongs to the Cohen class (\cite[Theorem 4.5.1]{Gro01}). These classes are not disjoint, and in \cite{Bayer} the matrix-Wigner transforms that are also in the Cohen class are characterized. On the other hand, there are matrix-Wigner transforms that are not in the Cohen class, as for instance the Ambiguity function, that in our study on quasicrystals plays an important role.

In \cite[Theorem 1]{bfgo} we proved that, if
$$
W(\mu)=\sum_{(r,s)\in R\times S} c_{r,s}\delta_{(r,s)}
$$
for a tempered distribution $\mu\in \mathcal{S}'(\mathbb{R}^d)$, and if $R$ and $S$ are uniformly discrete sets in ${\mathbb R}^d$, then $\mu$ and $\widehat{\mu}$ are measures with support contained in $R$ and $S$, respectively. The matrix-Wigner is also considered in \cite{bfgo}, but only in the case of dimension $d=1$; the case of dimension greater than $1$, for the general matrix-Wigner transform, presents much more difficulties.

In this paper we go further in the analysis of which information can be deduced on a distribution $\mu$ and its Fourier transform, from the structure of the matrix-Wigner transform of $\mu$. More precisely, we consider two different (disjoint) classes of matrix-Wigner transforms, proving the following results.

\begin{theorem}\label{thm-main1}
Let $T$ be an invertible $2d\times 2d$ matrix of the form
$$
T=\left(\begin{array}{cc} A_0 &B_0 \\ C_0 &D_0 \end{array}\right),\qquad\text{with\ inverse}\qquad T^{-1}=\left(\begin{array}{cc} A &B \\ C &D \end{array}\right).
$$
Let $\mu\in {\mathcal S}^\prime({\mathbb R}^d)$ satisfy
$$
W_T(\mu) = \sum_{(r, s)\in R\times S} c_{r,s} \delta_{(r,s)}
$$ with $R, S\subset {\mathbb R}^d$ u.d. sets (see Definition \ref{ud}) and suppose that the following conditions are satisfied:
\begin{itemize}
\item[{\rm (i)}] $\sup_{s\in S} |c_{rs}| < \infty$ for every $r\in R$ and $\sup_{r\in R} |c_{rs}| < \infty$ for every $s\in S;$
\item[{\rm (ii)}] $\det (B_0-D_0)\neq 0$;
\item[{\rm (iii)}] $\det(A+B)\neq 0$.
\end{itemize}
Then $\mu$ and $\widehat{\mu}$ are measures whose supports, $\Lambda$ and $\Sigma,$ are uniformly discrete. If furthermore $W_T$ is in the Cohen class, then $\Lambda\subseteq R$ and $\Sigma\subseteq S$.

\end{theorem}

\begin{theorem}\label{thm-main2}
Let $T$ be an invertible $2d\times 2d$ matrix of the form
$$
T=\left(\begin{array}{cc} A_0 &B_0 \\ C_0 &B_0 \end{array}\right),
$$
where $A_0$, $B_0$ and $C_0$ are $d\times d$ matrices. Let $\mu\in {\mathcal S}^\prime({\mathbb R}^d)$ satisfy
$$
W_T(\mu) = \sum_{(r, s)\in R\times S} c_{r,s} \delta_{(r,s)}
$$
where $R, S$ are discrete subsets of ${\mathbb R}^d$. Then $\mu$ and $\widehat{\mu }$ are measures whose supports, $\Lambda$ and $\Sigma,$ are uniformly discrete; moreover, there exist invertible $d\times d$ matrices $M$ and $N$ such that $\Lambda-\Lambda \subset M(R),\ \Sigma-\Sigma\subset N(S)$, where $M(R)$ indicates the linear application $M$ computed on the set $R$, and similarly for $N(S)$. If moreover $R$ or $S$ is uniformly discrete then 
\begin{equation}\label{eq:main-ambiguity}
\mu = \sum_{j=1}^N P_j \sum_{\lambda\in L + \theta_j}\delta_\lambda,
\end{equation}
where $L$ is a lattice, $\theta_j\in {\mathbb R}^d$ and $P_j(x)$ is a trigonometric polynomial.
\end{theorem}
In this case we have that $\mu$ is a Fourier quasicrystal in the sense of Lev and Olevskii (see \cite{olevskii1}).

We shall precise our results in the following sections; here we limit ourselves to some observations in order to underline the main novelties with respect to what is already known. First of all, the conditions on the matrix $T$ in the previous two theorems have no intersections, in the sense that if $T$ satisfies the hypotheses of Theorem \ref{thm-main1} then it does not satisfy the hypotheses of Theorem \ref{thm-main2} and vice versa. The matrix-Wigner transforms considered in Theorem \ref{thm-main1} include the classical Wigner transform, as well as all the matrix-Wigner transforms belonging to the Cohen class, and in these cases conditions (ii) and (iii) are automatically satisfied. Then, comparing Theorem \ref{thm-main1} with \cite[Theorem 1]{bfgo} we observe that in this paper we extend the results of \cite{bfgo} to a class of matrix-Wigner transforms, in arbitrary dimension, but with the additional hypothesis (i); in the case of the classical Wigner this hypothesis is not necessary (and also for matrix-Wigner transforms in dimension $1$, cf. \cite[Theorem 15]{bfgo}), but in the general case considered here we could not avoid it. Observe however that the hypotheses of \cite[Theorem 1]{bfgo} imply (at least in dimension $d=1$) that $\sup_{(r,s)\in R\times S} |c_{r,s}|<+\infty$ (see \cite[Corollary 6]{bfgo}); then condition (i) does not represent a limitation with respect to the class considered in \cite{bfgo}, and Theorem \ref{thm-main1} allows for a remarkable enlargement of the class of time-frequency representations which can be used to ``detect'' quasicrystal structures.  
Concerning Theorem \ref{thm-main2}, the hypotheses on $T$ exclude its application to the classical Wigner; on the other hand the class considered in Theorem \ref{thm-main2} has as main example the Ambiguity function (see \eqref{Ambigfunct}), and for this class we prove here a stronger and somehow surprising result with respect to \cite{bfgo}, since here we only suppose that $R$ and $S$ are discrete, obtaining as a consequence that the supports of $\mu$ and $\widehat{\mu}$ are uniformly discrete.

\section{Preliminary results}

We start by fixing some notations. For a $n\times n$ matrix $M$ and a set $K\subset\mathbb{R}^n$ we write $M(K):=\{M\cdot k : k\in K\}$, where $M\cdot k$ indicates the product of the matrix $M$ by the column vector $k$. Moreover, for $M$ invertible, we indicate by $P_M$ the linear change of variable operator defined on $\mathcal{S}'(\mathbb{R}^n)$ that on $\Phi\in\mathcal{S}(\mathbb{R}^n)$ acts as
$$
P_M(\Phi)(x)=\Phi(Mx),\quad x\in\mathbb{R}^n.
$$ For $n = 2d$ and $M = T$ as in the introduction we can write
$$
W_T(\mu) = {\mathcal F}_2\left(P_T(\mu\otimes\overline{\mu})\right),\ \mu\in{\mathcal S}'({\mathbb R}^{d}),$$ where ${\mathcal F}_2$ denotes the partial Fourier transform
$$
{\mathcal F}_2 F(x,\omega) = \int_{{\mathbb R}^d} F(x,t) e^{-2\pi i \omega t}\ dt,\ \ x,\omega\in \mathbb{R}^d,$$
with usual standard extension to $F\in \mathcal S'(\mathbb R^{2d})$, which through this paper will be assumed to be conjugate-linear functionals. As usual, the polarized form of the matrix-Wigner transform is
$$
W_T(\mu,\nu) = {\mathcal F}_2\left(P_T(\mu\otimes\overline{\nu})\right),\ \mu,\nu\in{\mathcal S}'({\mathbb R}^{d}).
$$

\par\medskip 

\begin{definition}\label{ud}
	{\rm
	By a {\it ``discrete''} set we mean a subset $S\subset {\mathbb R}^{n}$ lacking accumulation points in ${\mathbb R}^{n},$ so it is closed. A set $A\subset {\mathbb R}^n$ is {\it ``uniformly discrete''} (u.d. from now on) if there is $\delta > 0$ such that $|r-s| \geq \delta$ whenever $s,r\in A, s\neq r.$
	
	For a set $E\subset
	\mathbb R^{2d}=\mathbb R^d_x\times\mathbb R^d_\omega$ we indicate
	the projections on the $x$ and $\omega$-coordinates as:
	$$
	\begin{array}{c}
		\Pi_1(E)=\{x\in\mathbb R^d: \exists \, \omega\in\mathbb R^d
		\hbox{\
			such\ that \ } (x,\omega)\in E\},\\
		\Pi_2(E)=\{\omega\in\mathbb R^d: \exists \, x\in\mathbb R^d
		\hbox{\ such\ that \ } (x,\omega)\in E\}.
	\end{array}
	$$
	}
\end{definition}
We now discuss some properties concerning supports. Taking into account the role of the partial Fourier transform in the definition of the matrix Wigner transform, the following result will be useful. It can be obtained with the same argument as in \cite[Lemma 8]{bfgo}, so we omit the proof. 
	\begin{lemma}\label{lem:Proj}
	Suppose that $\Psi\in\mathcal S'(\mathbb R^{2d})$. If
	$\Pi_1\supp\Psi$ or $\Pi_1\supp\mathcal F_2 \Psi$ are discrete sets in
	$\mathbb R^d$, then $\Pi_1\supp \Psi=\Pi_1\supp\mathcal F_2 \Psi$
	(and therefore both are discrete).
\end{lemma}

Also, \cite[Proposition 9]{bfgo} can be formulated as follows.	
\begin{proposition}\label{prop:T}
	Let $M$ be a real invertible $n\times n$ matrix. Then for every $\Psi\in \mathcal S'(\mathbb R^n)$ we
	have $\supp(P_M(\Psi))=M^{-1}(\supp \Psi)$.
\end{proposition}

The following lemma will be crucial in the proof of Theorem \ref{thm-main2}:
\begin{lemma}\label{differ-discrete}
	Let $S\subset {\mathbb R}^{n}$ be a discrete set. If $\Sigma-\Sigma\subset S$ then $\Sigma$ is uniformly discrete.
\end{lemma}
\begin{proof}
	Since $0\in S$ and $S$ is discrete there is $\delta > 0$ such that $0 < |u| < \delta$ implies $u\notin S.$ In particular $|x-y|\geq \delta$ for every $x,y\in \Sigma$ with $x\neq y.$
\end{proof}

\begin{lemma}\label{lem:support-wignerT} Let $T$ be an invertible real matrix with $T^{-1} = \left(
	\begin{array}{cc}
		A & B \\
		C & D \\
	\end{array}
	\right),
	$ where $A, B, C, D$ are submatrices of dimension $d\times d,$ and let $\mu\in {\mathcal S}^\prime({\mathbb R}^d)$ satisfy that $\Pi_1\left(\supp W_T(\mu)\right)$ is a discrete subset of ${\mathbb R}^d.$ Then 
	\begin{equation}\label{Pisupp}
	\Pi_1\left(\supp W_T(\mu)\right) = A(\supp \mu) + B(\supp \mu).
	\end{equation}
Furthermore:

i) If $\det(A+B)\ne 0$, then $\supp\mu$ is discrete.

ii) If $B=-A$ and $\det A\ne0$, then $\supp \mu$ is uniformly discrete.

\end{lemma}
\begin{proof} We apply Lemma \ref{lem:Proj} to $\Psi = P_T(\mu\otimes\overline{\mu})$ to get 
$$
\Pi_1\left(\supp W_T(\mu)\right) = \Pi_1\left(\supp {\mathcal F}_2\left(P_T(\mu\otimes\overline{\mu})\right)\right) = \Pi_1\left(\supp P_T(\mu\otimes\overline{\mu})\right).$$ According to Proposition \ref{prop:T}, this set coincides with 
\begin{equation}\label{proj}
\Pi_1\left(T^{-1}\left(\supp\mu \times \supp\mu\right)\right) = A(\supp \mu) + B(\supp \mu).
\end{equation}

As $A(\supp\mu) + B(\supp\mu)\supseteq (A+B)(\supp\mu)$, we have that $(A+B)(\supp\mu)$ is discrete. If $\det(A+B)\ne 0$, the invertibility of $A+B$ implies that also $\supp\mu $ is discrete. If instead $B=-A$, from \eqref{proj} it follows that $A(\supp\mu) - A(\supp\mu)$ is discrete which, by Lemma \ref{differ-discrete}, implies $A(\supp\mu)$ uniformly discrete. If $A$ is invertible, we have that $\supp\mu$ is also uniformly discrete.
\end{proof}

We remark that condition (i) contains as primary examples all representations of the Cohen class, see \eqref{T-Cohen}-\eqref{T-Cohen-inv}, and in particular the Wigner transform. The most important example satisfying (ii) is the {\it ambiguity function}, which is defined, for $f\in L^2({\mathbb R}^d),$ as 
\begin{equation}\label{Ambigfunct}
\mathop{Amb}(f)(x, \omega)=\displaystyle \int_{{\mathbb R}^d}f(t+\frac{x}{2})\overline{f(t-\frac{x}{2})} e^{-2\pi i  \omega t}dt,\ \ x,\omega\in {\mathbb R}^d.
\end{equation}
Then $\mathop{Amb}(\mu) = W_T(\mu)$ for $\mu\in {\mathcal S}'({\mathbb R}^d)$ where $T = \left(\begin{array}{cc} \frac{1}{2}\ \Id &\Id \\ -\frac{1}{2}\ \Id &\Id \end{array}\right)$. This representation occurs naturally in radar applications, where is often called radar ambiguity function (see e.g. \cite[Chapter 4]{Cook67}, \cite{Jans97}).

The following proposition specifies a class of matrix-Wigner transforms for which we have uniform discreteness of $\supp\mu$. The main example is the ambiguity function. 

\begin{proposition} Let $A_0, C_0$ be $d\times d$ matrices with $\det(A_0-C_0)\neq 0$ and consider $T_{A_0,C_0} = \left(
	\begin{array}{cc}
		A_0 & \Id \\
		C_0 & \Id \\
	\end{array}
	\right).
	$ If $\Pi_1\left(\supp W_{T_{A_0,C_0}}(\mu)\right)$ is discrete then 
	\begin{equation}\label{Pisuppdiff}
	\Pi_1\left(\supp W_{T_{A_0,C_0}}(\mu)\right) = (A_0-C_0)^{-1}(\supp\mu - \supp\mu),
	\end{equation}
	and $\supp\mu$ is uniformly discrete. In particular for ambiguity function
	$$
	\Pi_1\left(\supp Amb(\mu)\right) = \supp\mu - \supp\mu.
	$$
\end{proposition}
\begin{proof}	
The equality \eqref{Pisuppdiff} follows from \eqref{Pisupp} of Lemma \ref{lem:support-wignerT} and the fact that 
$$
T_{A_0,C_0}^{-1} = 
	\left(
	\begin{array}{cc}
		(A_0-C_0)^{-1} & -(A_0-C_0)^{-1} \\
		-C_0(A_0-C_0)^{-1} & A_0(A_0-C_0)^{-1} \\
	\end{array}
	\right).
$$ 
As $A_0-C_0$ is invertible, the discreteness of $(A_0-C_0)^{-1}(\supp\mu - \supp\mu)$ implies that of $\supp\mu - \supp\mu$, which by Lemma \ref{differ-discrete} implies the uniform discreteness of $\supp\mu$.
\end{proof}

The previous results can be contrasted with the following counterexample which shows that not all the cases of discrete support of $\mu\in \mathcal S'(\mathbb R^d)$ can be detected by the projection of matrix-Wigner transforms.
It is based on the counterexample to Lagarias' conjecture contained in \cite{favorov}. 

\begin{example}{\rm We consider the uniformly discrete set $$\Lambda= {\mathbb Z}^2 \cup \left\{\left(\sqrt{2}m_1, m_2 + \frac{1}{2}\right):\ m_1, m_2\in {\mathbb Z}\right\}.$$ Then it is not possible to find $\mu\in{\mathcal S}^\prime({\mathbb R}^2)$ and an invertible $4\times 4$ matrix $T$ with $T^{-1} = \left(
	\begin{array}{cc}
		A & B \\
		C & D \\
	\end{array}
	\right),$ $\mbox{det}(A)\neq 0, \mbox{det}(B)\neq 0,$ such that $\supp\mu = \Lambda$ and $\Pi_1\left(\supp W_T(\mu)\right)$ is discrete. }
\end{example}
\begin{proof} On the contrary $A(\Lambda) + B(\Lambda)$ is a discrete set, which is equivalent to the fact that $\Lambda + (A^{-1}B)(\Lambda)$ is discrete. We now put
	$
	A^{-1}B = \left(
	\begin{array}{cc}
		a & b \\
		c & d \\
	\end{array}
	\right).$ Then, the following two sets are discrete
	$$
	M:= {\mathbb Z}^2 + \left(
	\begin{array}{cc}
		a & b \\
		c & d \\
	\end{array}
	\right){\mathbb Z}^2,\ \ N:= {\mathbb Z}^2 + \left(
	\begin{array}{cc}
		a & b \\
		c & d \\
	\end{array}
	\right)\Lambda_1,$$ where $\Lambda_1 = \left(
	\begin{array}{cc}
		\sqrt{2} & 0 \\
		0 & 1 \\
	\end{array}
	\right){\mathbb Z}^2 + \left(
	\begin{array}{c}
		 0 \\
		\frac{1}{2} \\
	\end{array}
	\right).$ We observe that $$M = \left\{\left(m+am'+bn',n+{\color{red}c}m'+dn'\right):\ m,n,m',n'\in {\mathbb Z}\right\}.$$ After taking $n'=0$ we conclude that 
	$$
	\left\{\left(m+am', n+cm'\right):\ m,n,m'\in {\mathbb Z}\right\}$$ is discrete. Then also the set of fractional parts of $\left\{m'(a,c):\ m'\in {\mathbb Z}\right\}$ is discrete. This implies that $a$ and $c$ are rational numbers, since otherwise the previous set of fractional parts is infinite and it has some accumulation point in $[0,1]^2.$ A similar argument, but using the fact that the set $N$ is discrete, allows us to conclude that also $\sqrt{2}a, \sqrt{2}c\in {\mathbb Q}.$ This is a contradiction since $A^{-1}B$ is invertible, which implies that at least one of the coefficients $a$ or $c$ is different from $0.$
\end{proof}

\begin{remark}{\rm An example of $\mu\in {\mathcal S}'({\mathbb R}^2)$ supported on $\Lambda$ is given in \cite{favorov} whose Fourier transform is also supported on a uniformly discrete set. However, for $T$ as in the previous Example, $\Pi_1\left(\supp W_T(\mu)\right)$ is not discrete.}
\end{remark}



In the following we write $T_0$ for the matrix
\begin{equation}\label{T0-matrix}
T_0=\left(\begin{array}{cc} \Id &\frac{1}{2}\Id \\[0.2cm] \Id &-\frac{1}{2}\Id \end{array}\right).
\end{equation}
$T_0$ is associated to the classical Wigner transform, in the sense that $W_{T_0}(\mu)=W(\mu)$ for every $\mu\in\mathcal{S}'(\mathbb{R}^d)$.

We start by proving the following technical lemmas.

\begin{lemma}\label{tech-lem1}
Let $\mu\in \Sch'(\R^d)$, $\Phi\in\Sch(\R^d)$, and let $T$ be an invertible $2d\times 2d$ matrix. Then
\begin{equation}\label{rel-W-WT}
\langle W(\mu),\Phi\rangle = |\det T|\langle W_T(\mu),\mathcal{F}_2 P_{T_0^{-1}T}(\mathcal{F}_2^{-1}\Phi)\rangle.
\end{equation}
\end{lemma}

\begin{proof}
Since $W(\mu)=W_{T_0}(\mu)=\mathcal{F}_2 P_{T_0}(\mu\otimes\overline{\mu})$ we have
\begin{equation*}
\begin{split}
\langle W(\mu),\Phi\rangle &=\langle \mu\otimes\overline{\mu},P_{T_0^{-1}}(\mathcal{F}_2^{-1}\Phi)\rangle \\
&=\langle \mu\otimes\overline{\mu},P_{T^{-1}}\mathcal{F}_2^{-1}\mathcal{F}_2 P_T P_{T_0^{-1}}(\mathcal{F}_2^{-1}\Phi)\rangle \\
&=|\det T|\langle \mathcal{F}_2P_{T}(\mu\otimes\overline{\mu}),\mathcal{F}_2 P_T P_{T_0^{-1}}(\mathcal{F}_2^{-1}\Phi)\rangle \\
&=|\det T|\langle W_T(\mu),\mathcal{F}_2 P_{T_0^{-1}T}(\mathcal{F}_2^{-1}\Phi)\rangle,
\end{split}
\end{equation*}
where we have used the facts that $P_M^{-1}=P_{M^{-1}}$ and $P_MP_N=P_{NM}$ for two matrices $M$ and $N$.
\end{proof}

\begin{remark}
	{\rm
	Incidentally we remark that the same computation as above yields that the relation between two matrix-Wigner transforms $W_{T_1}, W_{T_2}$ is expressed 
	by the formula 
	$$
	\langle W_{T_1}(\mu,\nu),\phi_1\otimes\phi_2\rangle = 
	|\det T_2| \ \langle W_{T_2}(\mu,\nu),W_{T_1^{-1}T_2}(\phi_1,\widehat{\overline{\phi_2}})\rangle,  
	$$
	where $\mu,\nu\in \mathcal S'(\mathbb R^d), \phi_1,\phi_2\in \mathcal S(\mathbb R^d)$. 
	}
\end{remark}

The following result is essentially contained in \cite{bfgo}. Since statement (ii) does not appear explicitly, we include a proof for the reader's convenience. 
\begin{lemma}\label{tech-lem2}
Suppose that $\mu$ is a tempered distribution of the form
$$
\mu=\sum_{r\in\Lambda}\sum_{|\alpha|\leq N} a_r^\alpha \delta_r^{(\alpha)},
$$
with $\Lambda\subset\R^d$ uniformly discrete, $N\geq 0$ and $a_r^\alpha\in\mathbb{C}$. Then:
\begin{itemize}
\item[(i)] For every  $\phi_1,\phi_2\in\Sch(\R^d)$ we have
\begin{equation*}
\begin{split}
&\langle W(\mu),\phi_1\otimes\phi_2\rangle = \\
&\ = \sum_{|\alpha|\leq N}\sum_{\alpha_1\leq\alpha} \sum_{|\beta|\leq N}\sum_{\beta_1\leq\beta} \lambda_{\alpha,\beta}^{\alpha_1,\beta_1} \sum_{r\in\Lambda}\sum_{s\in\Lambda} a_r^\beta \overline{a_s^\alpha}\overline{\phi_1}^{(\alpha_1+\beta_1)}\left(\frac{r+s}{2}\right) \overline{\widehat{\phi_2}}^{(\alpha-\alpha_1+\beta-\beta_1)}(s-r),
\end{split}
\end{equation*}
where
$$
\lambda_{\alpha,\beta}^{\alpha_1,\beta_1}=\binom{\alpha}{\alpha_1}\binom{\beta}{\beta_1}(-1)^{|\alpha|+|\beta|+|\alpha-\alpha_1|+|\beta-\beta_1|} \left(\frac{1}{2}\right)^{|\alpha_1|+|\beta_1|}.
$$
\item[(ii)] Let $N\geq 1$, $\gamma\in\mathbb{N}^d_0$ with $|\gamma|=2N.$ Then there exists $\varepsilon>0$ and $\psi\in\mathcal{D}(B_\varepsilon)$ such that for every $t\in \mathbb{R}$ with $t\geq 1$, writing $\phi_2(y)=\mathcal{F}^{-1}_{x\to y}\left(\psi(tx)\right)$, we have that for every $\phi_1\in \mathcal S(\mathbb R^d)$: 
	\begin{equation}\label{left-h-s}
		\langle W(\mu), \phi_1\otimes\phi_2\rangle =  t^{2N}\sum_{r\in\Lambda}\left(\sum_{(\alpha,\beta)\in F^d_\gamma}a_r^\beta\overline{a_r^\alpha}\right)\overline{\phi_1}(r),
	\end{equation}
	where
	\begin{equation}\label{Fdgamma}
		F_\gamma^d=\{(\alpha,\beta) : |\alpha|=|\beta|=N,\ \alpha+\beta=\gamma\}.
	\end{equation}
\end{itemize}
\end{lemma}

\begin{proof}
Point (i) is obtained by a direct computation; it is explicitly shown in \cite[page 7]{bfgo}.

Concerning point (ii), we fix $\gamma\in {\mathbb N}_0^d$ with $|\gamma| = 2N\geq 2.$ Since $\Lambda$ is uniformly discrete, there is $\varepsilon > 0$ such that $|r-s|\geq \varepsilon$ for every $r,s\in \Lambda, r\neq s.$ We choose $\psi\in {\mathcal D}\left(B_\varepsilon\right)$ real valued such that $\psi^{(\gamma)}(0) = 1, \psi^{(\alpha)}(0) = 0$ for all $\alpha\neq \gamma.$ For $t\geq 1$ we consider $\phi_2$ such that $\widehat{\phi_2}(x)=\psi(tx).$ Then $\widehat{\phi}_2^{(\gamma)}(0)= t^{2N}$ while $\widehat{\phi}_2^{(\alpha)}(0) = 0$ for any $\alpha\neq \gamma.$ We observe that the conditions
$$
\alpha - \alpha_1 + \beta - \beta_1 = \gamma,\ \ |\alpha|\leq N,\ \  |\beta|\leq N, \alpha_1\leq \alpha,\ \ \beta_1\leq \beta
$$
imply $|\alpha| = |\beta| = N, \alpha_1 = \beta_1 = 0.$ Then from point (i) we get
\begin{equation}\label{phi2-0}
\begin{array}{*2{>{\displaystyle}l}}
	\langle W(\mu), \phi_1\otimes\phi_2\rangle & = \sum_{\begin{array}{c}|\alpha|=|\beta|=N\\ \alpha+\beta = \gamma\end{array}}\sum_{r\in \Lambda}a_r^\beta\overline{a_r^\alpha}\overline{\phi_1}(r)\overline{\widehat{\phi_2}}^{(\gamma)}(0) \\ & \\ & =  t^{2N}\sum_{r\in\Lambda}\left(\sum_{(\alpha,\beta)\in F_\gamma^d}a_r^\beta\overline{a_r^\alpha}\right)\overline{\phi_1}(r),
\end{array}
\end{equation}
where $F_\gamma^d$ is given by \eqref{Fdgamma}. 
\end{proof}

\section{Proof of Theorem \ref{thm-main1}}

In order to prove Theorem \ref{thm-main1} we recall from \cite[Theorem 10]{Bayer} that $W_T$ belongs to the Cohen class if and only if $T$ is of the form
\begin{equation}\label{T-Cohen}
T_{\rm Cohen}=\left(\begin{array}{cc} \Id &E+\frac{1}{2}\Id \\[0.2cm] \Id &E-\frac{1}{2}\Id \end{array}\right)
\end{equation}
for a $d\times d$ matrix $E$. Observe that in this case
\begin{equation}\label{T-Cohen-inv}
T_{\rm Cohen}^{-1}=\left(\begin{array}{cc} \frac{1}{2}\Id-E &\frac{1}{2}\Id+E \\[0.2cm] \Id &-\Id \end{array}\right).
\end{equation}
and conditions (ii) and (iii) of Theorem \ref{thm-main1} are always satisfied.

\begin{proof}[Proof of Theorem \ref{thm-main1}]	
	We first observe that from Lemma \ref{lem:support-wignerT}
	\begin{equation}\label{A+B}
	R=A(\supp\mu)+B(\supp\mu)\supseteq (A+B)(\supp\mu),
	\end{equation}
	which implies that $(A+B)(\supp\mu)$ is uniformly discrete. Then from hypothesis (iii) we have that $\supp\mu$ is uniformly discrete. Observe that, in particular, if $W_T$ belongs to the Cohen class then from \eqref{T-Cohen-inv} we obtain $\supp\mu\subseteq R$. Writing $\Lambda=\supp\mu$, from the structure Theorem for tempered distributions, we have that $\mu$ is of the form
	\begin{equation}\label{mu-expr-1}
		\mu=\sum_{r\in\Lambda}\sum_{|\alpha|\leq N} a_r^\alpha\delta_r^{(\alpha)},
	\end{equation}
	for a uniformly discrete set $\Lambda\subset\R^d$.
	
	We now check that $\mu$ is a measure. We have to prove that $a_r^\alpha=0$ for every $\alpha\neq 0$ and $r\in\Lambda$. Let $\phi_1, \phi_2 \in {\mathcal S}({\mathbb R^d})$; we want to analyze the left and right-hand side of \eqref{rel-W-WT}, for $\Phi=\phi_1\otimes\phi_2$. In order to analyze the left-hand side, from \eqref{mu-expr-1} we can use Lemma \ref{tech-lem2}. We then fix $\gamma\in {\mathbb N}_0^d$ with $|\gamma| = 2N\geq 2.$ For any $t\geq 1$ we take $\phi_2\in\Sch(\R^d)$ as in Lemma \ref{tech-lem2}(ii), obtaining that $\langle W(\mu), \phi_1\otimes\phi_2\rangle$ is given by \eqref{left-h-s}.
	
	Concerning the right-hand side of \eqref{rel-W-WT}, we first observe that, from Lemma \ref{tech-lem2}(ii), $\widecheck{\phi}_2(\cdot)=\tilde{\psi}(t\cdot)$, where $\widecheck{\phi}_2$ denotes the inverse Fourier transform of $\phi_2$ and $\tilde{\psi}(x)=\psi(-x).$ Then, writing $M=T_0^{-1}T$ and choosing $\phi_1$ with compact support, we have
	\begin{equation}
		\begin{array}{*2{>{\displaystyle}l}}
			|\langle W(\mu),\phi_1\otimes\phi_2\rangle| & = |\mbox{det}\ T|\cdot \langle W_T(\mu),\mathcal{F}_2 P_M(\phi_1\otimes\tilde{\psi}(t\cdot))\rangle |\\
			&=|\mbox{det}\ T|\cdot\left| \sum_{(r,s)\in R\times S} c_{r,s} \overline{ \left(\mathcal{F}_2 P_M(\phi_1\otimes \tilde{\psi}(t\cdot))\right)(r,s) } \right|. \label{diff-cases}
	\end{array}\end{equation} 
	
	We observe that 
	$$
	M=\left(\begin{array}{cc} M_{11} &M_{12} \\ M_{21} &M_{22}\end{array}\right) = \left(\begin{array}{cc} \frac{1}{2}(A_0+C_0) &\frac{1}{2}(B_0+D_0) \\ A_0-C_0 &B_0-D_0\end{array}\right). 
	$$ Since $M_{22}=B_0-D_0$ is invertible by hypothesis, making the change $z= y + M_{22}^{-1}M_{21}r$ in the expression 
	$$
	\mathcal F_2 P_M(\phi_1 \otimes \tilde{\psi}(t\cdot))(r,s) = \int_{{\mathbb R}^d}\phi_1(M_{11}r+M_{12}y)\tilde{\psi}(tM_{21}r + tM_{22}y)e^{-2\pi i y s}\ dy
	$$ we have
	\begin{equation*}
		\begin{split}
			&\mathcal F_2 P_M(\phi_1 \otimes \tilde{\psi}(t\cdot))(r,s)= e^{2\pi i \left(M_{22}^{-1}M_{21}r\right)s} \\
			&\quad\times\int_{{\mathbb R}^d} \phi_1((M_{11}-M_{12}M_{22}^{-1}M_{21})r+M_{12}z) \tilde{\psi}(tM_{22}z)e^{-2\pi i zs}dz.
		\end{split}
	\end{equation*} Hence 
	$$
	\left|\mathcal F_2 P_M(\phi_1 \otimes \tilde{\psi}(t\cdot))(r,s)\right| = \left|\widehat{g_{r,t}}(s)\right|$$ where
	$$
	g_{r,t}(z) = \phi_1((M_{11}-M_{12}M_{22}^{-1}M_{21})r+M_{12}z) \tilde{\psi}(tM_{22}z).$$ A necessary condition for $g_{r,t}$ not to be identically null is the existence of some $u = M_{22}z\in \supp \tilde{\psi}\subset B_\varepsilon$ such that $$(M_{11}-M_{12}M_{22}^{-1}M_{21})r+M_{12}z\in \supp \phi_1,$$ hence 
	\begin{equation}\label{nonnull}
	(M_{11}-M_{12}M_{22}^{-1}M_{21})r\in M_{12}M_{22}^{-1}\left(B_\varepsilon\right) + \supp \phi_1.
	\end{equation} 
	Moreover, from \cite[Theorem 2.1]{Lu_Shiou} we have that $M_{11}-M_{12}M_{22}^{-1}M_{21}$ is invertible, which implies that 
	$$
	\left(M_{11}-M_{12}M_{22}^{-1}M_{21}\right)(R)$$ is uniformly discrete. 
	It follows that condition \eqref{nonnull} can be satisfied only by a finite number of $r\in R$.
	Consequently there is a finite subset $F\subset R$ (depending on $\phi_1$) such that $g_{r,t} = 0$ whenever $r\notin F,$ $t \geq 1.$ From (\ref{diff-cases}) we conclude
	\begin{equation}\label{estimate W}
	|\langle W(\mu),\phi_1\otimes\phi_2\rangle| \leq |\mbox{det}\ T|\cdot \displaystyle\sum_{r\in F} \sup_{s\in S}|c_{rs}|\sum_{s\in S}\left|\widehat{g_{r,t}}(s)\right|.
	\end{equation} 
	We check now that, for every fixed $r\in R,$ the series in the right hand side is convergent and 
	$$
	\sum_{s\in S}\left|\widehat{g_{r,t}}(s)\right| \leq Ct\ \ \forall t\geq 1$$ for some constant $C > 0$ depending only on $\phi_1$ and $\psi.$ We denote 
	$$
	f(z):=\phi_1((M_{11}-M_{12}M_{22}^{-1}M_{21})r+M_{12}z)\theta(z),\ \ h(z) := \tilde{\psi}(M_{22}z),$$ where $\theta$ is a compactly supported smooth function equal to $1$ on $M_{22}^{-1}(B_\varepsilon)$, which contains $\supp h(tz)$ for $t\ge 1$.
	Then 
	$$
	g_{r,t}(z) = f(z) h(tz)$$ (recall that $r\in F$ is fixed) and both functions $f$ and $h$ are in ${\mathcal D}({\mathbb R}^d).$  In what follows $C$ is a constant which is not necessarily the same at each appearance. We have
	$$
	\begin{array}{*2{>{\displaystyle}l}}
		\sum_{s\in S} \left|\widehat{g_{r,t}}(s)\right| & = \sum_{s\in S}\left|\left(\widehat{f}\ast\widehat{h(t\cdot)}\right)(s)\right| \leq \sum_{s\in S}\int_{{\mathbb R}^d} |\widehat{f}(u)| |\widehat{h(t\cdot)}(s-u)|\ du\\ & \\ & = \frac{1}{t^d}\sum_{s\in S}\int_{{\mathbb R}^d}|\widehat{f}(u)| |\widehat{h}(\frac{s-u}{t})|\ du \leq  \frac{C}{t^d}\sum_{s\in S} \int_{{\mathbb R}^d}|\widehat{f}(u)| \big(\frac{1}{1 + \frac{|s-u|}{t}}\big)^{d+1}\ du \\ & \\ & = C t\sum_{s\in S}\int_{{\mathbb R}^d}|\widehat{f}(u)|\big(\frac{1}{t+|s-u|}\big)^{d+1}\ du \\ & \\ & \leq C t\sum_{s\in S}\int_{{\mathbb R}^d}\big(\frac{1}{1+|u|}\big)^{d+1}\big(\frac{1}{1+|s-u|}\big)^{d+1}\ du \\ & \\ & = C t\sum_{s\in S}\int_{{\mathbb R}^d}\Phi(u)\Phi(s-u)\ du, 	
	\end{array} $$ where $\Phi(x) = \big(\frac{1}{1+|x|}\big)^{d+1}.$ With a standard procedure we obtain 
	$$
	\begin{array}{*2{>{\displaystyle}l}}
		\int_{{\mathbb R}^d}\Phi(u)\Phi(s-u)\ du & = \int_{|u|\geq \frac{|s|}{2}}\Phi(u)\Phi(s-u)\ du + \int_{|u|\leq \frac{|s|}{2}}\Phi(u)\Phi(s-u)\ du\\ & \\ & \leq 2\|\Phi\|_1 \big(\frac{2}{2+|s|}\big)^{d+1}.
	\end{array}$$

Since $S$ is uniformly discrete, we can select $\varepsilon>0$ with the property that all balls $B^s_\varepsilon$ with radius $\varepsilon$ and center $s\in S$ are disjoint.
For sufficiently large $N\in\mathbb N$, each ball $B^s_\varepsilon$, with $s\in S$, 
contains a point $\alpha\in \frac{1}{N}\mathbb Z^d$ such that $|\alpha|\le |s|$.  
Then 
$$
\sum_{s\in S}|\widehat{g_{r,t}}(s)|\le Ct\sum_{s\in S}\left( \frac{2}{2+|s|}\right)^{d+1} \le Ct \sum_{\alpha\in \frac{1}{N}\mathbb Z^d}\left( \frac{2}{2+|\alpha|}\right)^{d+1} =Ct.
$$

From estimate \eqref{estimate W} we then have 
	\begin{equation}\label{right-h-s-th1}
		|\langle W(\mu),\phi_1\otimes\phi_2\rangle| \leq C t \displaystyle\sum_{r\in F} \sup_{s\in S}|c_{rs}|,\end{equation} where $F\subset R$ is a finite set depending on 
    $\phi_1$ and the constant $C$ is independent on $t\geq 1.$

	Then from \eqref{left-h-s} and \eqref{right-h-s-th1} we obtain that for some constant $C > 0$ and every $t\geq 1$
	$$
	t^{2N}\left|\sum_{r\in\Lambda}\left(\sum_{(\alpha,\beta)\in F^d_\gamma}a_r^\beta\overline{a_r^\alpha}\right)\overline{\phi_1}(r)\right|\leq C t;
	$$
	taking limits as $t\to \infty$ we conclude 
	$$
	\sum_{r\in\Lambda}\left(\sum_{(\alpha,\beta)\in F^d_\gamma}a_r^\beta\overline{a_r^\alpha}\right)\overline{\phi_1}(r) = 0
	$$
	for every compactly supported $\phi_1\in {\mathcal S}({\mathbb R}^d)$. Since $\Lambda$ is discrete we deduce 
	$$
	\sum_{(\alpha,\beta)\in F^d_\gamma}a_r^\beta\overline{a_r^\alpha} = 0\ \ \forall r\in \Lambda.$$ Hence, by \cite[Lemma 5]{bfgo}, $a_r^\alpha = 0$ for any $r\in\Lambda, |\alpha|=N.$ Proceeding by recurrence we finally obtain that $\mu$ is a measure.
	
	In order to prove the result on $\widehat{\mu}$ we recall from \cite{Bayer} that
	\begin{equation}\label{WT-Fourier}
		W_T(\mu)(x,\omega)=|\det T|^{-1} W_L(\widehat{\mu})(\omega,-x)
	\end{equation}
	with
	\begin{equation}\label{WT-Fourier-matr}
		L=\left(\begin{array}{cc}\Id &0 \\ 0 &-\Id\end{array}\right) (T^{-1})^{\rm t} \left( \begin{array}{cc} 0 &\Id \\ \Id &0 \end{array}\right).
	\end{equation}
	We then have
	$$
	L=\left(\begin{array}{cc}\Id &0 \\ 0 &-\Id\end{array}\right) 
	\left(\begin{array}{cc} A^{\rm t} &C^{\rm t} \\ B^{\rm t} &D^{\rm t}\end{array}\right)
	\left( \begin{array}{cc} 0 &\Id \\ \Id &0 \end{array}\right) =
	\left(\begin{array}{cc} C^{\rm t} &A^{\rm t} \\ -D^{\rm t} &-B^{\rm t}\end{array}\right),
	$$
	and moreover
	$$
	L^{-1}=\left(\begin{array}{cc}0 &\Id \\ \Id &0\end{array}\right) 
	\left(\begin{array}{cc} A_0^{\rm t} &C_0^{\rm t} \\ B_0^{\rm t} &D_0^{\rm t}\end{array}\right)
	\left( \begin{array}{cc} \Id &0 \\ 0 &-\Id \end{array}\right) =
	\left(\begin{array}{cc} B_0^{\rm t} &-D_0^{\rm t} \\ A_0^{\rm t} &-C_0^{\rm t}\end{array}\right).
	$$
	Then it is easy to see that $L$ satisfies hypotheses (ii) and (iii); moreover, from \eqref{WT-Fourier} we have that
	$$
	W_L(\widehat{\mu})(x,\omega)=|\det T|\sum_{(s,r)\in S\times (-R)} c'_{s,r}\delta_{(s,r)}
	$$
	with $c'_{s,r}=c_{-r,s}.$ Then $W_L(\widehat{\mu})$ satisfies hypothesis (i), and so, repeating the same proof for $W_L(\widehat{\mu})$ we obtain that also $\widehat{\mu}$ is a measure supported in a uniformly discrete set. Finally, from \eqref{T-Cohen} and \eqref{WT-Fourier-matr} we have that if $W_T$ belongs to the Cohen class, then also $W_L$ is in the Cohen class; then, proceeding as before, if $W_T$ belongs to the Cohen class we have that $\supp\widehat{\mu}\subseteq S$.
\end{proof}
\begin{remark}{\rm Condition (ii) in Theorem 1 is necessary for the function $h$ that appears in 
		the proof of the theorem to be in the Schwartz class. It therefore plays a key role in 
		proving that $\mu$ is a measure. On the other hand, requiring that $T$ satisfies condition (iii) is needed to prove that $\supp \mu$ is uniformly discrete and it is equivalent to requiring that $L$ satisfies condition (ii).}
\end{remark}


A well-known form of the uncertainty principle states that the Wigner transform $W(\mu)$ of a non zero distribution $\mu\in \mathcal S'(\mathbb R^d)$ can not have compact support. As a by-product of the previous proof we can now easily show that this can be generalised to matrix-Wigner transforms under the hypothesis of Theorem \ref{thm-main1}, and strengthened in the sense that neither of the projections of $W_T\mu$ can be bounded if $\mu\neq 0$.      
\begin{proposition}
	In the hypothesis of Theorem \ref{thm-main1} if $R$ or $S$ is finite, then $\mu=0$.
\end{proposition}
\begin{proof}
	Let $L$ be the matrix defined in \eqref{WT-Fourier-matr}, as remarked in the final part of the proof of Theorem \ref{thm-main1}, both matrices $T$ and $L$ satisfy conditions (ii) and (iii) of the hypothesis of Theorem \ref{thm-main1}, so that we have 
	\begin{equation}
		\begin{array}{l}
			R = \Pi_1\left(\supp W_T(\mu)\right) = A(\supp\mu) + B(\supp\mu)\supseteq (A+B)(\supp\mu), \\
			S = \Pi_1\left(\supp W_L(\widehat\mu)\right) = B_0^t(\supp\widehat\mu) + (-D_0^t)(\supp\widehat\mu)\supseteq (B_0-D_0)^t(\supp\widehat\mu)
		\end{array}
	\end{equation}
	It follows that if $R$ is finite then $(A+B)(\supp\mu)$ is finite, which implies $\supp\mu$ is finite, as $A+B$ is a bijection. By the Paley-Wiener theorem, $\widehat \mu$ is then an analytic function. On the other hand $S$ is u.d. and $S \supseteq (B_0-D_0)^t(\supp\widehat\mu)$ implies that also $\supp\widehat \mu$ is u.d., as $(B_0-D_0)^t$ is a linear bijection. This is impossible unless $\mu=0$.
	
	Similarly if $S$ is finite, then $\supp\widehat\mu$ is finite and $\mu$ is an analytic function. The inclusion $R \supseteq (A+B)(\supp\mu)$, with $R$ u.d., 
	 implies that $\supp \mu$ is u.d., which is impossible unless $\mu=0$.
\end{proof}	


\section{Proof of Theorem \ref{thm-main2}}


Before we proceed with the proof of our last result we need the following observation on block matrices.

\begin{lemma}\label{block-matr} 
Let $Z$ be an invertible $2d\times 2d$ matrix, and write
$$
Z=\left(
\begin{array}{cc}
	Y & U \\
	V &W
\end{array}
\right), \qquad Z^{-1}=\left(
\begin{array}{cc}
	E & F \\
	G &H
\end{array}
\right),
$$
where $Y,U,V,W,E,F,G,H$ are $d\times d$ matrices. Then

{\rm (a)} The following equivalences hold 
\begin{equation}\label{BlockMat1}
	\begin{split}
		\det Y\ne0 &\Longleftrightarrow \det H\ne0,\\
		\det W\ne0 &\Longleftrightarrow \det E\ne0.
	\end{split}
\end{equation}
\begin{equation}\label{BlockMat2}
	\begin{split}
		\det U\ne0 &\Longleftrightarrow \det F\ne0,\\
		\det V\ne0 &\Longleftrightarrow \det G\ne0.
	\end{split}
\end{equation}

{\rm (b)} In the case $\det Y\ne0$, or equivalently $\det H\ne0$, we have:
	\begin{equation}\label{BlockMat3}
	U=-Y \Longleftrightarrow F=H.
	\end{equation}
(Note that statements similar to \eqref{BlockMat3} could be deduced for other couples of sub-matrices of $Z$ and $Z^{-1}$, but this is the only one that we shall need.)
\end{lemma}

\begin{proof}
(a) From \cite[Theorem
2.1]{Lu_Shiou} we have that
\begin{equation}\label{BlockMat1}
	\begin{split}
		\det Y\ne0 &\Longrightarrow \det H\ne0,\\
		\det W\ne0 &\Longrightarrow \det E\ne0.
	\end{split}
\end{equation}
Since
$$
Z_1=\left(
\begin{array}{cc}
	U & Y \\
	W &V
\end{array}
\right) \quad\Longrightarrow\quad Z_1^{-1}=\left(
\begin{array}{cc}
	G &H \\
	E &F
\end{array}
\right),
$$
applying \eqref{BlockMat1} to $Z_1$ we also have
\begin{equation}\label{BlockMat2}
	\begin{split}
		\det U\ne0 &\Longrightarrow \det F\ne0,\\
		\det V\ne0 &\Longrightarrow \det G\ne0.
	\end{split}
\end{equation}
Finally, from the trivial fact $(Z^{-1})^{-1}=Z$, we have that the implications in \eqref{BlockMat1} and \eqref{BlockMat2} are actually biimplications.

(b) From $ZZ^{-1}=\Id$ we have $YF+UH=0$. Then $U=-Y$ implies $F=H$ as $Y$ is invertible, and $F=H$ implies $U=-Y$ as $F$ is invertible.
\end{proof}

\begin{proof}[Proof of Theorem \ref{thm-main2}]
We first observe that the fact that $T$ is invertible, together with the particular form of $T$, implies that the submatrix $B_0$ is invertible. Now, by Lemma \ref{block-matr} the hypothesis 
\begin{equation}\label{T}
T=\left(\begin{array}{cc} A_0 &B_0 \\ C_0 &B_0 \end{array}\right),\qquad \det B_0\ne 0 
\end{equation}
on the invertible matrix $T$ is equivalent to 
\begin{equation}\label{T_0}
T^{-1}=\left(\begin{array}{cc} A &-A \\ C &D \end{array}\right),\qquad \det A \ne0.
\end{equation}

We can then apply Lemma \ref{lem:support-wignerT}\ (ii) which shows that $\supp\mu$ is uniformly discrete. 
Then, writing $\Lambda=\supp\mu$ we have that $\mu$ is of the form
\begin{equation}\label{mu-expression}
\mu=\sum_{r\in\Lambda}\sum_{|\alpha|\leq N} a_r^\alpha\delta_r^{(\alpha)},
\end{equation}
for a uniformly discrete set $\Lambda\subset\R^d$.

We want to prove that $\mu$ is a measure, i.e., $a_r^\alpha=0$ for every $\alpha\neq 0$ and $r\in\Lambda$. We proceed similarly as in the proof of Theorem \ref{thm-main1}; for $\phi_1, \phi_2 \in {\mathcal S}({\mathbb R^d})$ we analyze the left and right-hand side of \eqref{rel-W-WT}, for $\Phi=\phi_1\otimes\phi_2$. For the left-hand side we use Lemma \ref{tech-lem2}; exactly as in the proof of Theorem \ref{thm-main1} we fix $\gamma\in {\mathbb N}_0^d$ with $|\gamma| = 2N\geq 1$ and $\phi_2\in\Sch(\R^d)$ as in Lemma \ref{tech-lem2}(ii), obtaining that $\langle W(\mu), \phi_1\otimes\phi_2\rangle$ is given by \eqref{left-h-s}.

Concerning the right-hand side of \eqref{rel-W-WT}, we proceed in a different way with respect to the proof of Theorem \ref{thm-main1}. We first observe that, by \eqref{T0-matrix},
$$
T_0^{-1}=\left(\begin{array}{cc} \frac{1}{2}\Id &\frac{1}{2}\Id \\[0.2cm] \Id &-\Id \end{array}\right),
$$
and from the particular form of $T$ we get
$$
T_0^{-1}T=\left(\begin{array}{cc} \frac{1}{2}(A_0+C_0) &B_0 \\[0.2cm] A_0-C_0 &0 \end{array}\right) =: \left(\begin{array}{cc} A_1 &B_1 \\ C_1 &0\end{array}\right).
$$
We then compute
\begin{equation}\label{estimate-WT}
\begin{split}
\mathcal{F}_2 P_{T_0^{-1}T} (\phi_1\otimes\widecheck{\phi}_2)(x,\omega) &=\mathcal{F}_{t\to\omega} \left[ \phi_1(A_1 x+B_1t) \widecheck{\phi}_2(C_1x)\right] \\
&=\frac{1}{|\det B_1|} e^{2\pi i (B_1^{-1}A_1 x)\omega}\widehat{\phi_1}\left((B_1^{-1})^{\rm t}\omega\right) \widecheck{\phi}_2(C_1x),
\end{split}
\end{equation}
where $\widecheck{\phi}_2$ indicates the inverse Fourier transform of $\phi_2$, and $\det B_1\neq 0$ since $B_1=B_0$. Then, since $\widecheck{\phi}_2$ has compact support and is of the form $\widecheck{\phi}_2(x)=\psi(-tx)$, choosing also $\widehat{\phi_1}$ with compact support and using the fact that $W_T(\mu)$ is locally of order $0$ we obtain from \eqref{estimate-WT}
\begin{equation}\label{right-h-s}
\left|\langle W_T(\mu),\mathcal{F}_2 P_{T_0^{-1}T}(\phi_1\otimes\widecheck{\phi}_2)\rangle\right|\leq C_1 \|\widehat{\phi_1}\left((B_1^{-1})^{\rm t}\omega\right) \psi(-tC_1x)\|_\infty:=C,
\end{equation}
where $C$ is a constant independent of $t\geq 1$. Then, from Lemma \ref{tech-lem1}, \eqref{left-h-s} and \eqref{right-h-s} we obtain that for every $t\geq 1$
$$
t^{2N}\left|\sum_{r\in\Lambda}\left(\sum_{(\alpha,\beta)\in F^d_\gamma}a_r^\beta\overline{a_r^\alpha}\right)\overline{\phi_1}(r)\right|\leq C |\det T|.
$$
This cannot be satisfied for every $t\geq 1$ unless
$$
\sum_{r\in\Lambda}\left(\sum_{(\alpha,\beta)\in F^d_\gamma}a_r^\beta\overline{a_r^\alpha}\right)\overline{\phi_1}(r) = 0
$$
for every $\phi_1\in {\mathcal S}({\mathbb R}^d)$ such that $\widehat{\phi}_1$ is compactly supported. Then $$\nu:=\displaystyle\sum_{r\in\Lambda}\left(\sum_{(\alpha,\beta)\in F^d_\gamma}a_r^\beta\overline{a_r^\alpha}\right)\delta_r\in {\mathcal S}'({\mathbb R}^d)$$ is a tempered distribution which vanishes on smooth functions whose Fourier transform is compactly supported. Consequently, by density, $\nu$ vanishes on every function in ${\mathcal S}({\mathbb R}^d).$ Then we can proceed as in the proof of Theorem \ref{thm-main1} deducing that, since $\Lambda$ is discrete,
$$
\sum_{(\alpha,\beta)\in F^d_\gamma}a_r^\beta\overline{a_r^\alpha} = 0\ \ \forall r\in \Lambda.$$ Hence $a_r^\alpha = 0$ for any $r\in\Lambda, |\alpha|=N$ (\cite[Lemma 5]{bfgo}); proceeding by recurrence we finally obtain that $\mu$ is a measure.

In order to prove that $\widehat{\mu}$ is a measure with uniformly discrete support we use \eqref{WT-Fourier}, \eqref{WT-Fourier-matr}, and we observe that
from \eqref{T} we get
$$
L=\left(\begin{array}{cc}\Id &0 \\ 0 &-\Id\end{array}\right) 
\left(\begin{array}{cc} A^{\rm t} &C^{\rm t} \\ -A^{\rm t} &D^{\rm t}\end{array}\right)
\left( \begin{array}{cc} 0 &\Id \\ \Id &0 \end{array}\right) =
\left(\begin{array}{cc} C^{\rm t} &A^{\rm t} \\ -D^{\rm t} &A^{\rm t}\end{array}\right).
$$
Then $L$ has the same form as $T$, and so, repeating the same proof with $W_L(\widehat{\mu})$ we obtain that also $\widehat{\mu}$ is a measure supported in a uniformly discrete set $\Sigma$.

In the case that $S$ is uniformly discrete, also $\Sigma-\Sigma$ is u.d. and we can apply \cite[Theorem 3]{olevskii2} to obtain (\ref{eq:main-ambiguity}). If instead $R$ is u.d. then we apply \cite[Theorem 3]{olevskii2} to $\widehat{\mu}$ and then use Poisson's summation formula to obtain (\ref{eq:main-ambiguity}).
\end{proof}

We end by some final remarks on the two main results we have proved in this paper.

\begin{remark}{\rm 
\begin{itemize}
\item[(a)] Concerning Theorem \ref{thm-main1}, we know from \cite{bfgo} that for $d=1$, as well as for the Wigner in arbitrary dimension, it holds without extra conditions on the coefficients; the same proof of \cite{bfgo} can be repeated with no conditions on the coefficients for the $\tau-$Wigner in arbitrary dimension (see \cite{BDO2010} for the definition and main properties of the $\tau-$Wigner).
\item[(b)] We have already observed that the conditions on $T$ of Theorems \ref{thm-main1} and \ref{thm-main2} are disjoint; the matrix-Wigner representations in the Cohen class are all included in Theorem \ref{thm-main1}, see \eqref{T-Cohen}-\eqref{T-Cohen-inv}, while the Ambiguity function, obtained for
$$
T=\left(\begin{array}{cc}\frac{1}{2}\Id &\Id \\[0.2cm] -\frac{1}{2}\Id &\Id\end{array}\right),
$$
is included in Theorem \ref{thm-main2}.
\end{itemize}}
\end{remark}

\end{document}